\documentclass{psp}
\usepackage{amssymb,graphics,color,url}
\usepackage{amsmath,amsfonts,amstext,verbatim}

\newcounter{mycount}

\newenvironment{numlist}{\begin{list}{\arabic{mycount}.}%
   {\usecounter{mycount}\labelwidth=1cm\labelsep=3pt\itemsep 0pt}}{\end{list}}

\newtheorem{thm}{Theorem}
\newtheorem{lemma}[thm]{Lemma}
\newtheorem{prop}[thm]{Proposition}
\newtheorem{cor}[thm]{Corollary}

\numberwithin{equation}{section}

\numberwithin{figure}{section}
\numberwithin{thm}{section}

\renewcommand{\P}{\mathbb{P}}
\newcommand{\E}{\mathbb{E}}
\newcommand{\Z}{\mathbb{Z}}
\newcommand{\R}{\mathbb{R}}
\renewcommand\L{{\mathbb L}}
\newcommand\N{{\mathbb N}}
\newcommand\T{{\mathbb T}}
\newcommand{\df}{\textbf}
\newcommand\pc{{p_{\mathrm{c}}}}
\newcommand\pcc{p_{\mathrm c}}
\newcommand\pcs{\vec p_{\mathrm{c}}}

\newcommand\psurf{p_{\mathrm{surf}}}
\newcommand\pfin{p_{\mathrm{fin}}}
\newcommand\ptree{p_{\mathrm{fin}}(\T_b)}
\newcommand{\bin}{{\mathrm{bin}}}

\newcommand{\zz}{\widehat{\Z}}

\renewcommand{\o}{{\mathrm{o}}}
\renewcommand{\O}{{\mathrm{O}}}
\newcommand\rad{\mathrm{rad}}
\newcommand\CS{\mathcal {S}}
\newcommand\CT{\mathcal {T}}

\newcommand\la{\langle}
\newcommand\ra{\rangle}
\newcommand\be{\begin{equation}}
\newcommand\ee{\end{equation}}
\newcommand\comp[1]{#1^{\mathrm{c}}}
\newcommand\pdi{\Delta_n}
\newcommand\Si{\Sigma}
\newcommand\tpc{\widetilde p_{\mathrm{c}}}

\newcommand\Dee{\De_{\mathrm{e}}}

\newcommand\Om{\Omega}
\newcommand\De{\Delta}
\newcommand\om{\omega}
\newcommand\s{\sigma}
\renewcommand\a{\alpha}
\newcommand\g{\gamma}
\newcommand\eps{\epsilon}

\newcommand{\pd}{\partial}
\newcommand{\lra}{\leftrightarrow}

\newcommand\nlra{\nleftrightarrow}
\newcommand{\oo}{\infty}
\newcommand{\sm}{\setminus}
\newcommand{\ol}{\overline}
\newcommand{\es}{\varnothing}
\renewcommand{\emptyset}{\es}

\newcommand\qq{\qquad}
\newcommand\rsim[1]{\stackrel{#1}{\sim}}
\newcommand\resp{respectively}

\newnumbered{example}{Example}
\newunnumbered{notation}{Notation}
\newromanexpr\Hess{Hess}

\begin{document}

\title[Percolation of finite clusters and infinite surfaces]
{Percolation of finite clusters \\ and infinite surfaces}

\author[Grimmett, Holroyd, and Kozma]
{GEOFREY R.\ GRIMMETT\\
Statistical Laboratory,
Centre for
Mathematical Sciences,\addressbreak
University of Cambridge,
Wilberforce Road, Cambridge CB3 0WB, UK\addressbreak
email\textup{: \url{g.r.grimmett@statslab.cam.ac.uk}}\addressbreak
URL\textup{: \url{http://www.statslab.cam.ac.uk/~grg/}}
\nextauthor ALEXANDER E.\ HOLROYD\\
Microsoft Research, 1 Microsoft Way, Redmond, WA 98052, USA\addressbreak
email\textup{: \url{holroyd@microsoft.com}}\addressbreak
URL\textup{: \url{http://research.microsoft.com/~holroyd/}}
\nextauthor GADY KOZMA\\
Department of Mathematics, Weizmann Institute of Science,\addressbreak
POB 26, Rehovot 76100, Israel\addressbreak
email\textup{: \url{gady.kozma@weizmann.ac.il}}\addressbreak
URL\textup{: \url{http://www.wisdom.weizmann.ac.il/~gadyk/}}}

\receivedline{Received \textup{16} April \textup{2013}, Revised
\textup{1} October \textup{2013}}

\maketitle

\begin{abstract}
Two related issues are explored for bond percolation on $\Z^d$ (with $d \ge
3$) and its dual plaquette process. Firstly, for what values of the parameter
$p$ does the complement of the infinite open cluster possess an infinite
component? The corresponding critical point $\pfin$ satisfies $\pfin \ge
\pc$, and strict inequality is proved when either $d$ is sufficiently large,
or $d \ge 7$ and the model is sufficiently spread out. It is not known
whether $d \ge 3$ suffices. Secondly, for what $p$ does there exist an
infinite dual surface of plaquettes? The associated critical point $\psurf$
satisfies $\psurf \ge \pfin$.
\end{abstract}

\tableofcontents

%\keywords{Percolation, plaquette model, branching process.}
%\subjclass[2010]{60K35, 82B43}

\section{Introduction}\label{sec:intro}
Bond percolation on the square lattice $\Z^2$ has a natural dual process,
which is itself a bond percolation model.  This fact has contributed to a
detailed understanding of percolation in two dimensions, see for example
\cite{G99,Kes82,We07}.  The picture is more complicated in $d$ dimensions
with $d\geq 3$, in part because the natural dual model is a process on
\emph{plaquettes} rather than edges, and these plaquettes form
$(d-1)$-dimensional \emph{surfaces}.  Perhaps the first systematic study of
the plaquette process appeared in \cite{ACCFR}, where so-called area- and
surface-laws were proved.  Later papers dealing with plaquettes include
\cite{DeuP,GrHol,GH10}, and also \cite{LMMRS} on first-order phase transition
in the random-cluster model (see also \cite[Chap. 7]{G-RC}).

We study two related questions concerning bond percolation on $\Z^d$ and its
dual process, of which the first is as follows.   Suppose we remove all
vertices that lie in an infinite open cluster of a bond percolation process
with parameter $p$.  The set $X$ that remains is the union of the vertex sets
of all finite clusters, and it induces a subgraph of $\Z^d$.  We denote this
graph by $X$ also.  For what values of $p$ does $X$ possess an infinite
connected component with positive probability? We may define a critical
probability $\pfin$ such that almost surely $X$ has an infinite component for
$p<\pfin$, and not for $p>\pfin$. Let $\pc$ be the percolation critical
probability. Clearly $\pfin\geq\pc$, since $X=\Z^d$ a.s.\ for $p<\pc$.  In
$d=2$ dimensions, we have $\pfin=\pc$, since self-duality implies that for
$p>\pc$ the infinite open cluster contains cycles that enclose every vertex.
In $d=3$ dimensions, it is natural to expect the strict inequality
$\pc<\pfin$, since slightly above $\pc$ the infinite open cluster should not
be sufficiently dense to prevent connections in its complement.  We prove the
last inequality in high dimensions.

\begin{thm}\label{gady-intro}
For $d\geq 19$, we have the strict inequality $\pc<\pfin$.
\end{thm}

Our proof of Theorem~\ref{gady-intro} relies on the recent proof in
\cite{KN11} that the one-arm critical exponent $\rho$ takes its mean-field
value $\tfrac12$ in high dimensions.  Indeed, we prove that $\pc<\pfin$
provided $\rho<1$ (see Theorem~\ref{thm:koz} for the precise statement); this
is believed to be the case for all $d\geq 5$ but not for $d=3,4$. We do not
know whether or not $\pc<\pfin$ for $3\leq d\leq 18$.

Theorem~\ref{thm:koz} is stronger than Theorem~\ref{gady-intro} in two
further respects.  Firstly, $X$ can be replaced with the complement of the
infinite cluster of percolation on the spread-out lattice (while
connectedness in $X$ still refers to $\Z^d$). This enables
Theorem~\ref{gady-intro} to be extended to sufficiently spread-out lattices
for all $d\geq 7$.  Secondly, $X$ may be replaced with the set of vertices that
are not within distance $F$ of the infinite cluster, for 
any given $F<\oo$.

When $\pc < p < \pfin$, there is simultaneous occurrence of two disjoint
infinite objects, namely an infinite open cluster and an infinite component
of the non-percolating region. This is reminiscent of the result of Campanino
and Russo \cite{CamR} that \emph{site} percolation on $\Z^3$ with $p=\frac12$
contains both an infinite open and an infinite closed cluster.

Using fairly standard techniques, we show that $\pfin<1$ for $d\ge 2$, and
that an infinite component of $X$ (when it exists) is (a.s.) unique.  When
the lattice $\Z^d$ is replaced with a regular tree, $\pfin$ may be explicitly
computed (as done essentially in \cite{JJ99}) and it satisfies
$\pc<\pfin$.

\pagebreak
The above results lead to a partial answer to (and were in part motivated by)
our second main question, which concerns the plaquette process that is dual
to bond percolation on $\Z^d$.  A \df{plaquette} is a $(d-1)$-dimensional
face of a unit cube centred at a vertex of $\Z^d$. An edge $e$ of $\Z^d$
crosses a unique plaquette, called the \df{dual} of $e$. We declare a
plaquette \df{open} if and only if its dual edge is closed in the bond
percolation model. Thus the plaquettes form an i.i.d.\ percolation process
with parameter $1-p$. Open plaquettes can form surfaces, and one may ask
whether these surfaces undergo a phase transition at $\pc$, in the sense that
`infinite surfaces' exist for $p<\pc$ and not for $p>\pc$. Such a statement
is of course contingent on a precise definition of `infinite surface'.  We
prove that, according to one natural choice of definition, the phase
transition does \emph{not} occur at $\pc$ when $d$ is sufficiently large.

We call two plaquettes \df{adjacent} if their intersection is a
$(d-2)$-dimensional cube.  We say that a set of plaquettes is \df{connected}
if it induces a connected graph via this adjacency relation, and we say that
it has \df{no boundary} if every $(d-2)$-cube lies in an even number of
plaquettes 
(in other words, if it is a $(d-1)$-cycle in the homology over $\Z/2\Z$).
A \df{surface} is a connected set of plaquettes with no
boundary.  We define the critical probability $\psurf$ such that there exists
(a.s.)  an infinite surface of open plaquettes for $p<\psurf$, and not for
$p>\psurf$.

\begin{thm}\label{relation-intro}
For  $d\geq 2$, we have $\pfin\leq \psurf$.
\end{thm}

Theorems \ref{gady-intro} and \ref{relation-intro} have the following
immediate consequence.

\begin{cor}
For  $d\geq 19$, we have the strict inequality $\pc<\psurf$.
\end{cor}

Thus, infinite dual surfaces of plaquettes (as defined above) exist even
\emph{strictly} above $\pc$ in high dimensions.  When $d=2$, an infinite
surface is a connected union of doubly infinite dual paths, and therefore
$\psurf=\pc$ ($=\pfin$) in this case.  We do not know whether $\pc<\psurf$
for $3\leq d\leq 18$.

One may also impose further topological or other constraints on surfaces.  As
a preliminary result in this direction, we show in
Proposition~\ref{prop:surface} that, for $p>0$ sufficiently small, there
exists a surface of open plaquettes that is uniformly homeomorphic to a
hyperplane.  In dimension $3$, surfaces of zero genus (homeomorphic to
planes, spheres or discs) arise as the natural candidates for dual objects that block
\emph{entangled} bond-clusters.  However, there is additional complication
here: existence of (say) a sphere enclosing the origin in the complement of
the set of open bonds does not imply existence of a sphere enclosing the
origin composed of open plaquettes.  For details see \cite{GrHol}.

The above questions are presented more formally in Section \ref{sec:not}
below.  Section \ref{sec:top} contains a key topological lemma concerned with
the external boundary of a finite subset of $\Z^d$.  This is followed in
Sections \ref{sec:perc}--\ref{ssec:trees} by treatment of the union
$X$ of the
finite clusters, in the two cases of a lattice and a regular tree
respectively.  Section \ref{sec:surf} is devoted to infinite surfaces, and
their connection with $X$.

\section{Notation}\label{sec:not}

Let $d \ge 1$, and let $\Z^d$ be the set of $d$-vectors
$x=(x_1,x_2,\dots,x_d)$ of integers. A \df{facet} is a subset of $\Z^d$ of
the form $F=x+(A_1\times\dots\times A_d)$ where $x\in\Z^d$ and each $A_i$ is
either $\{0\}$ or $\{0,1\}$.  If $\{0,1\}$ appears $k$ times, we call $F$ a
$k$-dimensional facet, or a $k$-facet. We shall focus on $1$-facets and
$(d-1)$-facets, called respectively \df{edges} and \df{plaquettes}. The set
of edges is denoted $\E^d$, and the associated lattice is the graph
$\L^d=(\Z^d,\E^d)$.  An edge $\{x,y\}$ will also be written as the unordered
pair $\langle x,y \rangle$.

Two facets of different dimensions are said to be \df{incident} if one is a
subset of the other.  Two plaquettes $\pi_1$, $\pi_2$ are \df{adjacent},
written $\pi_1 \sim \pi_2$, if some $(d-2)$-facet is incident to both.  A set
$P$ of plaquettes is called \df{connected} if the graph with vertex set $P$
and adjacency relation $\sim$ is connected.  The \df{boundary} of a set $P$
of plaquettes is the set of all $(d-2)$-facets that are incident to an odd
number of elements of $P$.  A \df{surface} is a (finite or infinite)
connected set of plaquettes with empty boundary.

We make use of another notion of adjacency between plaquettes also, namely
$\pi_1 \rsim 1 \pi_2$ if some $1$-facet is incident to both.

We introduce the shifted (\lq dual') set $\zz^d:=\Z^d+h$, where
$h:=(\tfrac12,\tfrac12,\dots,\tfrac12)\in\R^d$.  A facet of $\zz^d$ is any
subset of the form $F+h$ where $F$ is a facet of $\Z^d$, and the concepts of
incidence, adjacency, connectedness, boundaries and surfaces are defined as
for $\Z^d$. For any $k$-facet $F$ of $\Z^d$, there is a unique $(d-k)$-facet
$F'$ of $\zz^d$ with the same centre of mass as $F$, and we call $F$ and $F'$
\df{duals} of one another.  (Equivalently: (i) $F'$ is the unique
$(d-k)$-facet of $\zz^d$ whose convex hull intersects the convex hull of $F$;
or (ii) if $F$ is expressed as $x+(A_1\times\dots\times A_d)$, then $F'$ is
obtained by replacing each occurrence of $\{0,1\}$ with $\{\tfrac12\}$, and
each occurrence of $\{0\}$ with $\{-\tfrac12,\tfrac12\}$.)  In particular,
the dual of an edge $e\in \E^d$ is a plaquette, which we denote $\pi(e)$. Let
$\Pi$ denote the set of all plaquettes of $\zz^d$.

We turn now to probability. Let $\Om=\{0,1\}^{\E^d}$ be endowed with the
product $\s$-field. Let $p\in [0,1]$, and write $\P_p$ for product measure on
$\Om$ with parameter $p$, and let $\E_p$ be the associated expectation
operator. For $\om\in\Om$, we call the edge $e$ \df{open} (\resp,
\df{closed}) if $\om(e)=1$ (\resp, $\om(e)=0$). The plaquette $\pi(e)$ is
declared \df{open} (\resp, \df{closed}) if $e$ is closed (\resp, open). Thus,
each plaquette is open with probability $1-p$. Percolation theory is
concerned with the structure of the connected components, or \df{open
clusters}, of the graph $(\Z^d,\eta(\om))$, where $\eta(\om)=\{e:\om(e)=1\}$
is the set of open edges of the configuration $\om$. Let $\pc=\pc(d)$ denote
the critical probability of bond percolation on $\L^d$, that is, the infimum
of  $p$ for which there is strictly positive probability of an infinite open
cluster. See \cite{G99} for a general  account of percolation.

We describe next the two events studied in this paper. Let
\be \CS:=\{\Pi
\text{ contains an infinite open surface}\}. \label{sdef} \ee
Since $\CS$ is a decreasing subset of $\Om$, and is invariant under
lattice-shifts, there exists $\psurf=\psurf(d) \in[0,1]$ such that \be
\P_p(\CS) =\,\begin{cases} 1 &\text{if } p <\psurf,\\
0 &\text{if } p>\psurf.
\end{cases}
\label{pcsurf}
\ee

A \df{path} of $\L^d$ is an alternating sequence $v_0,e_1,v_1,e_2,\dots$ of
distinct vertices $v_i$ and edges $e_i=\la v_{i-1},v_i\ra$. A path is called
\df{open} if all its edges are open. Let $x,y\in \Z^d$. We write $x \lra y$
if there exists an open path of $\L^d$ with endpoints $x$ and $y$. We write
$x\lra\oo$ if there exists an infinite open path with endpoint $x$. Let
$$
X:=\{x\in\Z^3: x \nlra \oo\},
$$
so that $X$ is the union of the vertex-sets of the finite open clusters
of the percolation process.  We turn $X$ into a graph by adding all edges
of $\E^d$ with both endpoints in $X$. Let
\be
\CT := \{X\text{ has an infinite connected component}\}.
\label{tdef}
\ee
As in the case of $\CS$, there  exists $\pfin=\pfin(d)\in[0,1]$ such that
\be
\P_p(\CT) =\,\begin{cases} 1 &\text{if } p <\pfin,\\
0 &\text{if } p>\pfin.
\end{cases}
\label{pcfin}
\ee

\section{Topological lemma}\label{sec:top}

This section contains a fundamental lemma concerning the external edge-boundary of
a finite subset of $\Z^d$. The case when $d=2$ appears essentially
in \cite[App.]{Kes82}. The lemma is closely related to results for $d$ dimensions
of \cite{DeuP,GrHol,Timar} and perhaps elsewhere.
The proof given here makes use of \cite[Thm 7.3]{G-RC}.

Let $A \subseteq \Z^d$, and let the \df{external boundary} $\De A$ be the set
of vertices of $\Z^d \setminus A$ that (i) are adjacent to some element of
$A$, and (ii) lie in some infinite path of $\L^d$ not intersecting $A$. The
(external) \df{edge-boundary} $\Dee A$ is the set of edges $e=\la x,y\ra$
such that $x \in A$ and $y\in \De A$. For $A\subseteq\Z^d$, we define
$$\Pi(A):=\{\pi(e): e\in \Dee A\},$$
the set of plaquettes
dual to edges in its edge-boundary. Let $|A|<\oo$. A set $\Si\subseteq\Pi$
of plaquettes is said to \df{separate
$A$ from infinity} if  every infinite path starting in $A$ contains some edge $e$ with
$\pi(e) \in \Si$.  Similarly, a set $\Si$ of plaquettes of $\Z^d$ is said to \df{separate
$A$ from infinity} if every infinite path starting in $A$ contains some vertex incident to some
element of $\Si$.

\begin{lemma}\label{topolem}
If $A\subset\Z^d$ is finite and connected then $\Pi(A)$ is a surface.
\end{lemma}

\begin{proof}
We show first that $\Pi(A)$ is connected.
Let $W_1,W_2,\dots,W_k$ be the finite, connected components
of $\Z^d\sm A$, and let
$$
\ol A:= A\cup\left(\bigcup_{i=1}^k W_k\right).
$$
We call the $W_k$ the \df{holes} of $A$. It is seen as follows that $\ol A$
has no holes. Suppose $w$ lies in a hole of $\ol A$. Then any infinite path
$\gamma$ from $w$ has a last vertex $f(\gamma)$ lying in $\ol A$. It must be
that $f(\gamma) \in A$, since $f(\gamma)\in W_i$ would contradict the
definition of the $W_i$. Therefore, $w \in \ol A$, a contradiction.

We prove two facts about the boundaries of $A$ and $\ol A$. Firstly,
\be
\Pi(\ol A) = \Pi(A).
\label{gen0}
\ee
That $\Pi(A) \subseteq \Pi(\ol A)$ follows from the definition of $\Pi(A)$.
Conversely, if $\pi(e) \in \Pi(\ol A) \sm \Pi(A)$, then $e$ has one endvertex denoted $a$
in $\ol A \sm A$, and another that is joined to infinity off $\ol A$. Thus $a$ lies in a hole
of $A$, in contradiction of the fact that $a$ is joined to infinity off $A$.

Let $\ol \Pi(\ol A)$ be the set of plaquettes
whose dual edges have exactly one endvertex in $\ol A$.
Evidently, $\Pi(\ol A) \subseteq \ol\Pi(\ol A)$, and we claim that
\be
\Pi(\ol A) = \ol \Pi(\ol A).
\label{gen1}
\ee
If, on the contrary,
$\pi(e) \in \ol\Pi(\ol A) \sm \Pi(\ol A)$, then
$e$ has an endvertex lying in a hole of $\ol A$.
Since $\ol A$ has no holes, \eqref{gen1} follows.

%\pagebreak
By \cite[Thm 7.3]{G-RC}, there exists a subset $Q
\subseteq \ol \Pi(\ol A)$ that separates $\ol A$ from $\oo$ and is connected.
By \eqref{gen0}--\eqref{gen1}, $Q \subseteq \Pi(A)$.
Also, $\Pi(A) \subseteq Q$, since if there exists $\pi\in \Pi(A)
\sm Q$, then $Q$ does not separate $A$ from $\oo$. Therefore,
$\Pi(A)=Q$, implying that $\Pi(A)$ is connected.

Next we show that $\Pi(A)$ has empty boundary. Write $u_j$ for a unit vector
in the direction of increasing $j$th coordinate. Let $e\in \Dee A$, and
assume without loss of generality that $e$ has the form $\la a, a+u_1\ra$
with $a\in A$; the other cases are treated in the same way after rotation of
$\Z^d$. It suffices to take $a=0$. The plaquette corresponding to the edge
$e=\la 0,u_1\ra$ is $\pi(e) = \{\frac12\} \times \{-\frac12,\frac12\}^{d-1}$,
and is incident to $2(d-1)$ $(d-2)$-facets of $\zz^d$, of which we shall
consider $f:=\{\frac12\}^2\times\{-\frac12,\frac12\}^{d-2}$; the other
such $(d-2)$-facets may be treated similarly.

The facet $f$ is incident to exactly four plaquettes in $\Pi$, namely the $\pi(e_i)$ with
$e_1=e=\la 0,u_1\ra$, $e_2=\la u_1,u_1+u_2\ra$, $e_3=\la u_1+u_2,u_2\ra$,
$e_4=\la u_2,0\ra$. As we proceed around the cycle $e_1,e_2,e_3,e_4$,
we encounter vertices that either lie in $A$ or are connected to infinity off $A$.
Each time we pass from a vertex in one category to a vertex in the other, we traverse a plaquette
in $\Pi(A)$. Hence we traverse an even number of such plaquettes, and therefore
$f$ is incident to an even subset of $\Pi(A)$.
\end{proof}

\section{Percolation of finite clusters}\label{sec:perc}

Let $G$ be an infinite, locally finite graph,
and consider bond percolation on $G$
with parameter $p$. Let $X$ be the set of vertices
lying in no infinite cluster of the process. The subgraph of $G$ induced by $X$ is also denoted $X$.
We consider the question
of whether $X$ has an infinite connected component.  The case of $\Z^d$
is treated in this section, and the corresponding issue for a regular tree
is considered briefly in the next section.

Recall the definition \eqref{pcfin} of the critical point $\pfin=\pfin(d)$
for the hypercubic lattice $\L^d$.
When $p<\pc$, all clusters are (a.s.)\ finite, so that
$X= \Z^d$. Hence, $\pc\le \pfin$.

The main goal of this section is to prove Theorem \ref{gady-intro}, which follows from Theorem \ref{thm:koz} below,
by the main result of \cite{KN11}.  Later in this section we also prove that $\pfin<1$ (Theorem~\ref{pfin_less}),
and that $X$ has at most one infinite component (Theorem~\ref{bk-thm}).

In order to state Theorem \ref{thm:koz} we introduce next some further notation.
Let $\|\cdot\|$ denote $L^\oo$
distance on $\Z^d$.  Let $S,F \in \{0,1,2,\dots\}$; the parameter $S$ is the range of a spread-out model,
and $F$ is a `fattening' parameter. Consider the \df{spread-out} percolation
model on $\Z^d$ in which edges are placed between any pair $x, y \in\Z^d$ with $0<\|x-y\|\le S$,
and each edge is declared independently open with probability $p$. The corresponding probability measure is denoted $\P_p^S$,
and $\pcc^S$ denotes the critical point.

Let $I$ be the set of vertices that lie in infinite open clusters.
Write
$$
I^F:= \{x\in \Z^d: \|x-y\| \le F \text{ for some } y \in I\},
$$
and
let $X^F = \Z^d \setminus I^F$. With $X^F$ considered as a subgraph of
$\L^d$, we let $\CT^F$ be the event that $X^F$ has an infinite
connected component, and
\begin{equation}\label{g45}
\pfin^{S}(F) := \sup\{p: \P_p^S(\CT^F) >0\}.
\end{equation}
Let $\rad(C)$ denote the radius of the open
cluster $C$ at the origin,
\be\label{def-rad}
\rad(C) := \sup\{\|x\|: 0 \lra x\}.
\ee

\pagebreak
\begin{thm}\label{thm:koz}
Let $d \ge 2$ and $S,F \ge 0$.  There exists $c>0$ such that the following holds. Suppose
that, for all large $n$, say $n \ge N(c)$,
\begin{equation}\label{assume0}
\P_{\pcc^S}^S(\rad(C) \ge n) \le \frac cn.
\end{equation}
Then $\pcc^F < \pfin^{S}(F)$.
\end{thm}

\begin{proof}[Proof of Theorem \ref{gady-intro}]
When $d\geq 19$ and $S=F=0$, \eqref{assume0} is proved in \cite{KN11}.
\end{proof}

The one-arm critical exponent $\rho=\rho(d)$ is given by
\be\label{rho-def}
\P_{\pcc^S}^S(\rad(C) \ge n) \approx n^{-1/\rho},
\ee
where various interpretations
of the symbol $\approx$ are possible. See for example \cite[Sect.\ 9.1]{G99}
for a discussion of critical exponents and universality.  A relation of the form
\eqref{rho-def} is believed to hold in a wide variety of settings including
the spread-out models on $\Z^d$ for all $d\ge 2$.  It is known that $\rho(d)=\frac12$ for
the nearest-neighbour model with $d \ge 19$, and
for the sufficiently spread-out model in $7$ and more dimensions \cite{KN11}.  On the
other hand it is believed that $\rho(2)= \frac{48}5$ ($> 1$),
so that \eqref{assume0} is expected to fail in two dimensions, as indeed does the
conclusion of Theorem \ref{thm:koz}. It is proved in \cite{LSW02} (see also \cite{SW01}) that
$\rho$ exists for site percolation on
the triangular lattice and satisfies $\rho(2)=\frac{48}5$,
and in \cite[\S 5]{Kes87}  that $\rho(2) \ge 3$, or more specifically
\begin{equation}\label{bk}
\P_{\pc}(\rad(C) \ge n) \ge cn^{-1/3},
\end{equation}
for bond percolation on the square lattice.  The argument leading to \eqref{bk}
may be extended (using \cite[Thm 5.1]{Kes82} and \cite[eqn (6.56)]{G99})
to obtain a lower bound of order
$n^{-(d-1)/2}$ on $\L^d$ with $d \ge 2$ (for $d=2$ the bound is worse
than \eqref{bk}, but is valid also in the spread-out case).

The critical exponents $\rho$, $\eta$ are expected to satisfy
the scaling inequality $2\le \rho(d-2+\eta)$, with equality when $d \le 6$
(see \cite[Sect.\ 9.1]{G99} for the definition
of $\eta$ and a discussion of the scaling relations,
and \cite{Tas} for a proof of the inequality under reasonable assumptions).
Assuming the equality for $d \le 6$, it follows that
$\rho<1$  if and only if
\begin{equation}\label{g105}
\eta > 4-d.
\end{equation}
Numerical studies of \cite{BFMSPR,LZ} (see also
\cite[Table 4.1]{Hughes2})
suggest that \eqref{g105} fails
when $d=3$ but holds when $d=5$. The evidence when $d=4$, while less conclusive,
suggests that \eqref{g105} fails in this case also.
It thus seems possible that \eqref{assume0}  fails when $d=3,4$
and holds when $d=5$.

Theorem \ref{thm:koz} will be proved using a block argument based on a bound
for a typical cluster-radius.
A similar technique has been used recently in \cite{ADKS} to
study a forest-fire problem in seven and more dimensions,
also subject to the assumption  $\rho<1$.

\begin{proof}[Proof of Theorem \ref{thm:koz}] Consider 
first the unfattened nearest-neighbour model with $S=F=0$; at the end of
the proof, we indicate the necessary
changes for the general case.
Assume that \eqref{assume0} holds for $c>0$ and $n \ge N(c)$. Rather then
deleting the set of points in \emph{infinite} open paths, we shall delete a
larger set that is specified in terms of certain \emph{finite} connections. A
box argument will then be used to show the existence of an infinite component
in the remaining graph.

We specify first the set of points to be deleted. Let $n\in\N$, and write $B_n = (-n, n]^d \cap \Z^d$
and $\pd B_n = B_n \setminus B_{n-1}$.
Let $R_n$ be the set of vertices in $B_n$ joined by open paths to $\pd B_{2n}$,
and note that $R_n$ is a superset of the subset of $B_n$ containing
points lying in infinite open paths.
By \eqref{assume0},
\begin{equation}\label{K1}
\E_\pc|R_n| \le c(2 n)^d  n^{-1} =c2^d n^{d-1}.
\end{equation}
By Markov's inequality,
\begin{equation}\label{L2}
\P_\pc(|R_n| \ge \tfrac12 |B_n|) \le \frac{2c}n.
\end{equation}

Let $C_1,C_2,\dots,C_m$ be the components of $B_n$ after deletion of
$R_n$. Each vertex in the external boundary $\De C_i$ of some $C_i$ lies either outside
$B_n$ or in $R_n$. Each vertex of $R_n$ is in at most $2d$ such external boundaries.
Therefore,
\begin{equation}\label{K2}
2d |R_n| \ge \sum_{i=1}^m |\pdi C_i|,
\end{equation}
where $\pdi C := \De C \cap B_n$.
By \cite[Thm 8]{BL91}, there exists $K>0$ such that,
for any connected subset $C$ of $B_n$ with $|C| \le \frac45 |B_n|$,
we have $|\pdi C| \ge K |C|^{(d-1)/d}$.

Let $L_n$ be the event that there exists $i$ with $|C_i| \ge \frac45 |B_n|$,
and let $M_n = \{|R_n| < \frac12 |B_n|\}$. On the event $\comp{L_n} \cap
M_n$, by \eqref{K2},
\begin{align}\label{K3}
2d |R_n| &\ge K\sum_{i=1}^m |C_i| ^{(d-1)/d}\\
&\ge K \left( \sum_{i=1}^m |C_i| \right)^{(d-1)/d}\nonumber\\
&= K\left(|B_n| - |R_n|\right)^{(d-1)/d}
\ge \tfrac12 K |B_n|^{(d-1)/d}.
\nonumber
\end{align}
We take expectations to obtain by \eqref{K1} that
\begin{align*}
c2^dn^{d-1} &\ge \E_\pc\bigl(|R_n| ;\comp {L_n}\cap M_n\bigr)\\
&\ge \tfrac14(K/d) (2n)^{d-1} \P_\pc(\comp{L_n}\cap M_n).
\end{align*}
By \eqref{L2},
\begin{equation}\label{K4}
\P_\pc(L_n) \ge 1- \frac {8cd}K - \frac{2c}n.
\end{equation}
On the event $L_n$, we pick a $C_i$ of largest size, and we colour its
vertices \df{$0$-green}.

Inequality \eqref{K4} may be combined with a block argument to show the
existence of an infinite component in $X$, when $c$ and $n$ are chosen
suitably. For $z\in\Z^d$, define the block $B_n(z) := B_n + nz$. Two blocks
are designated \df{adjacent} if and only if they have non-empty intersection.
Let $R_n(z)$ and $L_n(z)$ be given as above but relative to the block
$B_n(z)$. The block $B_n(z)$  is called \df{good} if  $L_n(z)$ occurs. On
$L_n(z)$, we pick a largest component in the complement of $R_n(z)$, and
colour its vertices \df{$z$-green}. Let  $\Gamma$ be the set of all vertices
that are $z$-green for some $z$.

It may be seen that the set of random block-colours is a
$3$-dependent family of random variables
(see \cite{LSS} for a definition of $k$-dependence). 
By \cite[Thm 0.0]{LSS}, there
exists $\tpc\in(0,1)$ such that any $3$-dependent site percolation model on
$\Z^d$ with all site-marginals exceeding $\tpc$ has (a.s.)\ an infinite
cluster. Choose $c>0$ and $n \ge N(c)$ such that
\begin{equation*}
1- \frac{8cd}K - \frac{2c}n > \tpc.
\end{equation*}
By \eqref{K4} and the continuity in $p$ of $\P_p(L_n)$, we may find
$p>\pc$ such that $\P_p(L_n) > \tpc$. There exists, $\P_p$-a.s., an infinite connected
component  of good blocks.

Let $z,z'\in \Z^d$ be adjacent, and note that the intersection of $B_n(z)$ and $B_n(z')$ has cardinality $\tfrac12 |B_n|$.
Suppose $B_n(z)$ and $B_n(z')$ are both good.
Since each contains a component of green
vertices of size at least $\frac45 |B_n|$, at least $\frac 15$ of the vertices
in the intersection $B_n(z) \cap B_n(z')$ are both $z$-green and $z'$-green.
Therefore, the green sets in $B_n(z)$ and $B_n(z')$ have non-empty intersection.
It follows that there exists, $\P_p$-a.s., an infinite component in
$\Gamma$. Since no vertex of $\Gamma$ lies in an infinite
open cluster, the theorem is proved.

Essentially the same proof is valid for general $S$, $F$.
The set $R_n$ is defined as the set of $x \in B_n$ such that there
exists $y$ with $\|y-x\|\le F$ and $y$ is joined by an open path
to some vertex in $\Z^d \setminus B_{2n+2F-1}$. Equations
\eqref{K1}--\eqref{L2} are replaced by
\begin{align*}
\E^S_{\pcc^S}|R_n| &\le c(2n+2F)^d n^{-1},\\
\P_{\pcc^S}^S(|R_n| \ge \tfrac12 |B_n|) &\le \frac{2c}n \left(1+ \frac Fn\right)^d,
\end{align*}
and the proof then proceeds as before.
\end{proof}

\begin{thm}\label{pfin_less}
Consider bond percolation on the lattice $\L^d$ with
$d \ge 3$. We have that $\pfin < 1$.
\end{thm}

The proof below may be made quantitative, in that it provides a calculable
bound $p'<1$ with $\pfin\le p'$. This bound is however too imprecise to be
interesting.

\begin{proof}[Proof of Theorem \ref{pfin_less}]
The idea is as follows.
When $p$ is sufficiently close to $1$,
not only is there an infinite open cluster,
but this cluster is `fat' in the sense that it separates
the origin from infinity. In order to construct the
required cut-surface, we shall first use Lemma \ref{topolem} to
show that the origin is a.s.\ separated from infinity
by a surface in $\L^d$ with the property that every constituent edge is open.

Let $0 <\a < \pc$, and consider bond percolation
on $\L^d$ with parameter $\a$. Let $A \subset \Z^d$
be finite and connected, and let $C=C_A$ be the set of all vertices reached from $A$ by open paths.
Since
the percolation process is subcritical,
$C$ is (a.s.) finite and connected.
By Lemma \ref{topolem}, $\Pi(C)$
is a surface that separates $C$ (and hence $A$ also) from infinity.

Interchanging the roles of $\Z^d$ and $\zz^d$, the above conclusion may be
restated as follows. Suppose that each plaquette of $\Z^d$ is independently
declared \df{occupied} with probability $q$.  If $q>1-\pc$, a.s.\ every
finite subset of $\zz^d$ is separated from $\infty$ by a finite occupied
surface of plaquettes of $\Z^d$.  If an infinite path $\nu$ of
$\zz^d$ contains an edge dual to some plaquette $\pi$, the shifted path
$\nu +h$ contains a vertex incident to $\pi$.  Therefore, if $q>1-\pc$,
every finite subset of $\Z^d$ is separated from $\infty$ by a finite occupied
surface of plaquettes of $\Z^d$.

Returning to bond percolation on $\E^d$, call a plaquette of $\Z^d$ \df{good}
if every incident edge of $\E^d$ is open. Write $\g(\pi)=1$ (\resp,
$\g(\pi)=0$) if $\pi$ is good (\resp, not good). The set of plaquettes of
$\Z^d$ may be considered as the vertex-set of a graph $G$ with adjacency
relation $\rsim 1$ given in Section~\ref{sec:not}. The vector
$\g=(\g(\pi):\pi\in \Pi+h)$ is $1$-dependent in that, for any subsets
$\Pi_1$, $\Pi_2$ of $\Pi+h$ separated by graph-theoretic distance at least
$2$, the sub-vectors $(\g(\pi): \pi\in \Pi_1)$ and $(\g(\pi): \pi\in \Pi_2)$
are independent.  Let $p' \in (1-\pc,1)$. By \cite[Thm 0.0]{LSS} (see
also \cite[Thm 7.65]{G99}), there exists $p''<1$ such that: when $p>p''$,
the law of $\g$ stochastically dominates product measure with parameter $p'$.

\pagebreak
Let $x \in \Z^d$. Let $E_x$ be the event that $x$ is separated from infinity
by some finite surface of good plaquettes of $\Z^d$, and that in addition
every vertex incident with this surface lies in the (a.s.)\ unique infinite
open cluster. By the above, for $p > p''$,
\begin{equation*}
1-\P_p(E_0) \le \P_{p'}(\comp{S_n}) + \P_p(B_n \nlra\oo)
\to 0\qquad\text{as } n\to\oo,
\end{equation*}
where $S_n$ is the event that there exists a finite open surface of plaquettes of $\Z^d$
separating $B_n := (-n,n]^d \cap \Z^d$ from infinity.
Therefore, $\P_p(E_0)=1$. By translation-invariance,
\begin{equation*}
\P_p\Bigl(\bigcap_x E_x\Bigr) =1.
\end{equation*}
On the last event, every infinite path from $X$ intersects the
infinite open cluster. It follows that $X$ has a.s.\ no infinite component for $p''< p<1$.
Therefore, $\pfin\le p''$ as required.
\end{proof}

\begin{thm}\label{bk-thm}
Let $p\in[0,1]$. The set $X$ has either a.s.\ no infinite component or a.s.\ a
unique infinite component.
\end{thm}

The proof of Theorem~\ref{bk-thm} will be an adaptation of the celebrated
argument of Burton and Keane \cite{BK}, and will proceed via the next two
lemmas.  The first describes the behaviour of the set $X$ under modifications
to the percolation configuration (however, it is easily seen that $X$ does \emph{not}
possesses the so-called `finite-energy property').  The second
conveniently encapsulates a sufficiently general consequence of the
`encounter point' argument of \cite{BK}.

For a percolation configuration $\omega\in\{0,1\}^{\E^d}$ and an edge
$e\in\E^d$, we write $\omega^e$ (\resp, $\omega_e$) for the
configuration that agrees with $\omega$ on $\E^d\setminus\{e\}$ and has
$\omega^e(e)=1$ (\resp, $\omega_e(e)=0$).

\begin{lemma}\label{modify}
For any configuration $\omega$ and any edge $e$ we have
$$X(\omega_e)=X(\omega^e)\cup F$$
for some (possibly empty) finite set $F=F(\omega,e)\subset\Z^d$.
\end{lemma}

\begin{proof}
The edge $e$ lies in some open cluster $C$ of $\omega^e$.  When the state of $e$
is changed from open to closed,
either the set of infinite open clusters is unchanged (in
which case we take $F=\emptyset$), or $C$ breaks into an infinite and a finite
cluster (in which case we take $F$ to be the  vertex-set of the latter).
\end{proof}

The number of \df{ends} of a connected graph is the supremum over its finite
subgraphs of the number of infinite components that remain after removing the
subgraph.  The following is proved (in the greater generality of amenable
transitive graphs) in \cite[remark following Cor.~5.5]{blps-group}; see
also \cite[Ex.~7.28]{lyons-peres-book}.

\begin{lemma}\label{ends}
Let $H$ be a random subset of $\E^d$ that is invariant in law under
translations of $\Z^d$.  Then, almost surely, no component of $H$ has more than $2$ ends.
\end{lemma}

\begin{proof}[Proof of Theorem~\ref{bk-thm}]
Let $N$ be the number of infinite components of $X$.  Since $N$ is a
translation-invariant function of a collection of independent random
variables, it is a.s.\ equal to some constant $n$.

We first show that $n\in\{0,1,\infty\}$.  Suppose on the contrary that
$1<n<\infty$.  There exist $x,y\in\Z^d$ that lie in
distinct infinite components of $X$ with positive probability.  Let $U$ be
the vertex set of a finite path in $\Z^d$ connecting $x$ and $y$.  On the
event mentioned above, modify the configuration $\omega$ by making every edge
incident to $U$ closed.  In the modified configuration $\omega'$, all
vertices of $U$ lie in $X(\omega')$.  By Lemma~\ref{modify}, $X(\omega')\supseteq
X(\omega)$, so $X(\omega')$ has an infinite component that contains the
original infinite components of $x$ and $y$.  By Lemma~\ref{modify} again,
$X(\omega')\setminus X(\omega)$ is finite, so no new infinite components have
been created, and thus $X(\omega')$ has strictly fewer that $n$ infinite
components. Since the modification involved only a fixed finite set of edges,
it follows that $X(\omega)$ has fewer than $n$ components with positive
probability, a contradiction.

We employ a similar argument to eliminate the possibility $n=\infty$.  If
$n=\infty$, there exist $x,y,z \in \Z^d$ that lie in distinct
infinite components of $X$ with strictly positive probability.
On this event, we can modify $\omega$ by making
a finite set of edges closed in such a way that $x$, $y$ and $z$ now lie in a single
infinite component of $X$.  By Lemma~\ref{modify}, only finitely many
vertices are added to $X$ in this process (and none are removed), so the
resulting component has at least $3$ ends.  Therefore $X$ has a component
with at least $3$ ends with positive probability, in contradiction of
Lemma~\ref{ends}.
\end{proof}

\section{On regular trees}\label{ssec:trees}

We consider briefly the question of $\pfin$ for percolation
on a regular tree. Rather than the $(b+1)$-regular tree, for convenience we work
with the rooted tree $\T_b$ all of whose
vertices other than the root (denoted $0$) have degree $b+1$; the root has degree $b$.
Let
$$
\rho_b(p) := \P_p(X \mbox{ possesses an infinite cluster}).
$$
Since $X$ is decreasing in the natural coupling of the processes
as $p$ varies, $\rho_b$ is a non-increasing function. As before,
$\rho_b$ takes only
the values $0$ and $1$, and the two `phases' are
separated by the critical value
$$
\ptree = \inf\{p: \rho_b(p)=0\}.
$$
The following theorem is based on an elementary calculation using the
theory of branching processes. The result is implicit in \cite[Sect.\ 3.3]{JJ99},
and an extension is found at \cite[Thm 2.2]{GJ05}.

\begin{thm}\label{tree}
Let $b \ge 2$. Then
$$
\ptree = \frac{b^{b/(b-1)}-b}{b^{b/(b-1)}-1}, \quad\text{and}\quad \rho_b(\pfin)=0.
$$
\end{thm}

In particular, $\pfin(\T_2) = \frac23$, and
$$
\ptree = (1+\o(1))\frac1b\log b\qq\text{as } b \to\oo.
$$

We outline below a different proof from those of \cite{GJ05,JJ99}. The
following additional consequence is explained after the proof:
\begin{equation}\label{extra1}
\P_{\pfin-\eps}\bigl(\text{the $X$-component of $0$ is infinite}\bigr)  = c_b \eps + \O(\eps^2)
\end{equation}
as $\eps\downarrow 0$, for some $c_b \in (0,\oo)$ which may
be computed explicitly.

\begin{proof}[Proof of Theorem \ref{tree}]
The bond percolation cluster $C$ at the root amounts to a branching
process with family-sizes distributed as $\bin(b,p)$ and
(probability) generating function $G(s)=(1-p+ps)^b$. We assume $p>1/b$,
since $X=T$ a.s.\ otherwise.
Conditional on extinction, the family-size generating function is
$H(s) = G(s\eta)/\eta$ where
the probability $\eta=\eta(p)$ of extinction is the smallest
non-negative root of $G(s)=s$.

Consider a branching process with
generating function $H$. The generating function $T(s)$ of
the total size satisfies
\be
T(s) = sH(T(s)),
\label{v-1}
\ee
so that $T'(1) = H(1) + H'(1)T'(1)$, and hence
$$
T'(1) = \frac 1{1-H'(1)} = \frac 1{1-G'(\eta)}.
$$

Let $D$ be a finite connected subgraph of $T$,
rooted at some vertex $v$ and with the property that every vertex
other than $v$ has generation number
strictly greater than that of $v$. A \df{boundary edge} of $D$ is an edge
$\la x,y\ra$ of $T$ such that $x\in D$, $y \notin D$, and $y$ is a child of $x$.
If $D$ is \emph{infinite}, it is considered
to have \emph{no} boundary edges.

Returning to the original $\bin(n,p)$ branching process, we consider
the embedded branching process of boundary edges
of rooted finite clusters. The cluster $C$ at the origin has
$1+|C|(b-1)$ boundary-edges if finite (and $0$ if infinite),
with generating function
\be
K(s) = 1-\eta + \eta sT(s^{b-1}).
\label{v0}
\ee
A branching process with generating function $K$
survives with strictly positive probability if and only if
$K'(1)>1$, which is to say that
$$
\eta [1+(b-1)T'(1)] > 1,
$$
or equivalently
\be
G'(\eta) > \frac{1-b\eta}{1-\eta}.
\label{v1}
\ee

Since $G(s) = (1-p+ps)^b$,
\be
G'(\eta) = bp(1-p+p\eta)^{b-1} = \frac{bp\eta}{1-p+p\eta},
\label{v5}
\ee
and \eqref{v1} becomes
\be
\eta>\frac{1-p}{b-p}.
\label{v2}
\ee

Suppose that \eqref{v2} were to hold with equality. Since $\eta=G(\eta)$,
$$
\left(1-p+\frac{p(1-p)}{b-p}\right)^b = \frac{1-p}{b-p},
$$
which may be solved to find that
$$
p=\frac{b^{b/(b-1)}-b}{b^{b/(b-1)}-1}.
$$
It follows that \eqref{v2} holds if and only if $p$ is
strictly smaller than the above value.

In summary, the non-percolating part of the
original tree percolates if and only if
$$
p < \frac{b^{b/(b-1)}-b}{b^{b/(b-1)}-1}.
$$
That $\rho(\pfin)=0$ follows from the fact that a
critical branching process with non-zero variance dies out almost surely.
\end{proof}

Let $\kappa(p)$ be the probability that the root lies in an infinite component of $X$.
By an elementary analysis based on the above, one may compute the critical
exponent of $\kappa$ as $p\uparrow \pfin$.

Let $b\ge 2$ and $p> 1/b$. By differentiating \eqref{v-1},
$$
T'(1) =
\frac 1{1-G'(\eta)}, \quad T''(1) =\frac{2G'(\eta)(1-G'(\eta))+\eta G''(\eta)}
{(1-G'(\eta))^3},
$$
where $\eta$ is the extinction probability given earlier. 
The term $G'(\eta)$ is given in \eqref{v5}, and by a similar calculation
$$
G''(\eta) = \frac{p^2\eta b(b-1)}{(1-p+p\eta)^2}.
$$
The point $p=\pfin$ is
characterized by equality in \eqref{v1} and \eqref{v2},
and one may use Taylor's theorem to expand $T(s)=T_p(s)$ with
$p$ near $\ptree$ and $s$ near $1$.
With $K$ as in \eqref{v0},
the survival probability $\kappa(p)$ is the largest root in $[0,1]$
of the equation $1-\kappa = K(1-\kappa)$. In conclusion, we obtain 
\eqref{extra1}.

\section{Infinite open surfaces}\label{sec:surf}

We return to the lattice $\L^d$ with $d \ge 3$, and consider
the existence (or not) of an infinite open surface,
that is, an infinite connected set of
open plaquettes with empty boundary. Recall the event $\CS$ of \eqref{sdef},
and the critical probability $\psurf$. We indicate first the elementary
inequality  $\psurf >0$.  Later in this section we prove the inequalities $\pc\leq\psurf$ and $\pfin\leq \psurf$.

\begin{prop}\label{prop:surface0}
For $d \ge 2$, we have $\psurf > 0$.
\end{prop}

A stronger result holds.  We identify a plaquette with the $(d-1)$-dimensional cube in $\R^d$ that is the closed convex hull of its $2^{d-1}$ points, and we identify a surface with the union of the cubes corresponding to its plaquettes.
The surface $S$
is said to be \df{uniformly homeomorphic to a hyperplane}
if there exists a bijection from  $S$ to the
hyperplane $\{0\} \times\R^{d-1}$ that is uniformly continuous with uniformly continuous inverse.
Let $\mu=\mu_d$ be the connective constant of $\L^d$ (see \cite{GH10} for a definition of $\mu$).

\begin{prop}\label{prop:surface}
Let $d \ge 2$. For $p<\mu^{-2}$,
there exists (a.s.)\ an infinite open surface that is uniformly homeomorphic to a hyperplane.
\end{prop}

Although this implies Proposition \ref{prop:surface0}, we include a short proof
of the first proposition also.

\begin{proof}[Proof of Proposition \ref{prop:surface0}]
We call a $d$-facet $B$ of $\zz^d$ \df{good} if each of
its $2d$ incident plaquettes is open, and we write $\g(B)=1$ if $B$ is good, and
$\g(B)=0$ otherwise. As in the proof of Theorem \ref{pfin_less}, the cubes of
$\zz^d$ may be considered as the vertices of a graph, and there exists
$r<\oo$ such that their states are $r$-dependent. By \cite[Thm 0.0]{LSS},
there exists $p' \in (0,1)$ such that: for $p \le p'$, the vector $\g$
stochastically dominates product measure with parameter $\pcs$, the
critical probability of oriented site percolation on $\L^d$.

Therefore, for $p < p'$, there exists a.s.\ an infinite oriented path
$v_1,v_2,\dots$ of $\Z^d$ such that the cubes
$v_i+\{-\frac12,\frac12\}^d$ are good. The set $V=\{v_1,v_2\dots\}$
generates an infinite surface $\Pi(V)$ of open plaquettes.
(Note that $\Pi(V)$ is homeomorphic to a hyperplane, but not uniformly.)
\end{proof}

\begin{proof}[Proof of Proposition \ref{prop:surface}]
This relies on the methods developed in \cite{GH10},
and we include only a sketch of the proof.

For $x=(x_1,x_2,\dots,x_d) \in \R^d$, let $s(x) = \sum_i x_i$.
Define the hyperplane  $H=\{x\in \Z^d: s(x) =0\}$ of $\R^d$.
For $r \ge 0$, let $H_r = \{y\in \Z^d: s(y) = r\}$,
and  $H_+ =\bigcup_{r\ge 0} H_r$.

A path
$v_0,v_1,\dots, v_k$ of $\L^d$ is called \df{good} if (i) $v_i \in H_+$ for all $i$,
and (ii) for every $i$ with $s(v_{i-1}) < s(v_i)$, the edge
$\langle v_{i-1},v_i\rangle$ is open.  For $x \in H_0$, let
$K_x$ be the set of all $y \in H_+$ such that there exists a good path
from $x$ to $y$, and let $K=\bigcup_{x\in H_0} K_x$.

Let $p< \mu^{-2}$ and $\a \in (\mu p,1)$.
By an adaptation of the proof of \cite[Lemma 4]{GH10},
there exists $C=C(p,\a)<\oo$ such that, for $x \in H_0$,
\begin{equation}\label{g102}
\sum_{y \in H_r}\P_p(y \in K_x) \le  C\a^r, \qquad r \ge 1.
\end{equation}
As in \cite[eqn (2)]{DDGHS}, for $y \in H_+$,
\begin{align}\label{g103}
\P_p(y \in K) &\le \sum_{x\in H_0} \P_p(y \in K_x)\\
&= \sum_{z\in H_{s(y)}} \P_p(z \in K_0)
\le C \a^{s(y)}, \nonumber
\end{align}
by \eqref{g102}.

We now follow the proof of \cite[Thm 1]{GH10}. Let
$u,v\in H_+$ be such that $e=\langle u,v\rangle$ is an edge of
$\Z^d$ with  $u \in K$ and $v \notin K$.
Then $s(u) < s(v)$ and $e$ is necessarily closed, so that the plaquette $\pi(e)$ is open.
The set $S$ of all such plaquettes forms a surface.
The required homeomorphism $\phi: S \to H$ is given by the projection
onto $H$,
\begin{equation*}
\phi(z) = z - \ol z e,
\end{equation*}
where $\ol z= (z_1+z_2+\dots+z_d)/d$ and $e =(1,1,\dots,1)$.
\end{proof}

We indicate next that $\psurf \ge \pc$, and we defer the proof until later in this section.

\begin{thm}\label{sub-crit-surf}
For $d \ge 3$, we have $\pc\leq \psurf$.
\end{thm}

If the complement of the infinite open cluster contains an infinite
component, one may deduce the existence of an infinite surface. This is the
content of Theorem \ref{relation-intro}, which is a consequence of the following more
general proposition.

\begin{prop}\label{surflem}
\label{sep}
  Let $A\subset\Z^d$ be infinite and connected, and suppose that
its complement $\comp{A}$ has
an infinite component.  There exists an infinite surface of
plaquettes that are dual to edges of $\L^d$ having one vertex in $A$ and the other in $\comp{A}$.
\end{prop}

\begin{proof}[Proof of Theorem \ref{relation-intro}]
If $\pc=\pfin$, the claim is a trivial consequence of Theorem
\ref{sub-crit-surf}. Assume that $\pfin > \pc$. Let $\pc<p<\pfin$, and apply
Proposition \ref{surflem} to the vertex-set $A$ of the infinite open cluster.
\end{proof}

There follows a preliminary lemma that will be useful in the remaining proofs.
The convergence of this
lemma is in the product topology on $\{0,1\}^\Pi$.
That is, for a sequence $\Pi_n$ of subsets of $\Pi$, we write $\Pi_n \to \Pi_\oo$
if every $\pi\in\Pi_\infty$ lies in all but finitely many $\Pi_n$
while every $\pi\not\in\Pi_\infty$ lies in only finitely many $\Pi_n$.

\begin{lemma}\label{lem:prel}
Let $W_1\subseteq W_2\subseteq\cdots$ be an increasing sequence of
subsets of $\Z^d$ that are connected and finite and satisfy
$|W_n|\to\oo$ as $n\to\oo$. The limit 
$\Pi_\oo=\lim_{n\to\oo}\Pi(W_n)$ exists and has empty boundary and no finite components.
\end{lemma}

If the set $\Pi_\oo$ of Lemma \ref{lem:prel} is non-empty, then it possesses only
infinite components, and each such component has empty boundary.
Here is an example for which  $\Pi_\oo$ is non-empty.
Let
$W_n = \{(0,w)\in\Z\times\Z^{d-1}: \|w\|_{d-1} \le n\}$.
Then $\Pi_\oo$ comprises two infinite components.

\begin{proof}
Let $\Pi_n=\Pi(W_n)$. Since each $W_n$ is finite and connected in $\L^d$,
by Lemma \ref{topolem}, the $\Pi_n$ are connected and have empty boundaries.

We claim first that $\Pi_\infty:=\lim_{n\to\infty} \Pi_n$ exists in the sense of the
product topology. More specifically, we claim that, for any plaquette $\pi\in\Pi$,
exactly one of the following holds:
\begin{numlist}
\item $\pi\notin \Pi_n$ for all $n$,
\item there exists $k$ such that $\pi\in \Pi_n$ if and only if $n \ge k$.
\item there exist $k$, $m$ satisfying $k<m$ such that $\pi\in \Pi_n$ if and only if $k\le n < m$.
\end{numlist}
To prove this, we must show that, as the sequence $\Pi_n$ is
revealed in sequence, if $\pi$ appears,
it may be removed, but if so it never reappears.
For given $\pi\in\Pi$, let $k=\min\{n: \pi\in \Pi_n\}$ and assume $k<\oo$.
Thus $\pi=\pi(e)$ for some $e=\la v,w\ra$ with $v \in W_k$ and
$w$ joined to infinity off $W_k$. Either $\pi\in \Pi_n$ for all $n \ge k$, or
$m=\inf\{n>k:\pi\notin\Pi_n\}$ satisfies $m < \oo$.
In the latter case, $w$ lies either in $W_m$ or
in a hole of $W_m$ (that is, a finite connected component of
$\Z^d \sm W_m$). Therefore, for $n\ge m$,
$w$ lies in either $W_n$ or a hole of $W_n$. In either case $\pi\notin \Pi_n$,
and the claim is shown.

Let $f$ be a $(d-2)$-facet of $\zz^d$.  The collection of subsets
$F\subset\Pi$ such that $f$ lies in an even number of members of $F$ is a
cylinder subset of $\{0,1\}^\Pi$. Since every $(d-2)$-facet lies in an even
number of plaquettes in every $\Pi_n$, and $\Pi_n \to \Pi_\oo$, $\Pi_\oo$ has
empty boundary. By the same argument, $\Pi_\oo$ has no finite component (in
the plaquette graph with adjacency relation $\sim$).
\end{proof}

\begin{proof}[Proof of Theorem \ref{sub-crit-surf}]
We shall show that $p \le \psurf$ for all $p<\pc$.
Let $\om=(\om(e): e\in \E^d) \in \Om$ be such that
\be
\text{every open cluster of $\om$ is finite.}
\label{gen5}
\ee
From $\om$, we construct an increasing sequence $V_1,V_2,\dots$ of vertex-sets of $\Z^d$
as follows.
Let $V_1$ be the vertex-set of the open cluster $C_0$ at $w(0) := 0$.
Suppose we have constructed $V_1,V_2,\dots,V_n$, and each is finite.
Let $v=(v_1,v_2,\dots,v_d)$ be a rightmost vertex of $V_n$, in that $v_1 \ge x_1$
for all $x \in V_n$.
Let $w(n+1)=v+u_1$ where $u_1=(1,0,0,\dots,0)$,
and let $C_{w(n+1)}$ be the vertex-set of the open cluster of $w(n+1)$. Let
$V_{n+1}= V_n \cup C_{w(n+1)}$. Note that $V_{n+1}$ is
finite, and $|V_n| \to \oo$ as $n\to\oo$.

We apply Lemma \ref{lem:prel} to the increasing sequence $(V_n)$
to obtain the limit set $\Pi_\oo=\Pi_\oo(\om)$ of plaquettes.
The proof is completed by showing that, for $p<\pc$,
\be
\P_p(\Pi_\oo \ne \es) =1.
\label{gen6}
\ee
Let $p<\pc$, so that \eqref{gen5} holds almost surely.
Let $L$ be the singly-infinite line $[(-\oo,0]\times\{0\}^{d-1}]\cap\Z^d$.
Since $\Pi_n$ is connected and separates $V_n$ from infinity,
there exists an edge $f_n =\la -r-1,-r\ra\in L$ such that $\pi(f_n)\in \Pi_n$, and we pick
$r=r_n$ maximal with this property. Now, $f_n \ne f_{n+1}$ only if
$C_{w(n+1)} \cap L \ne \es$.
However,
$$
\sum_{n = 0}^\oo \P_p(C_{w(n+1)}\cap L \ne \es)
\le \sum_{n = 0}^\oo \P_p(\rad(C_0) \ge n) = \E_p(\rad(C_0)),
$$
where $\rad(C_0) = \sup\{\|x\|: 0\lra x\}$ is
the radius of $C_0$ as in \eqref{def-rad}.

If $p < \pc$, we have that 
$\E_p(\rad(C_0)) < \oo$; see \cite{AB,Men2,Men1},
and also \cite[Chap.\  5]{G99}.
By the Borel--Cantelli lemma, a.s.\ only finitely many of
the $w(n)$ are connected by open paths to $L$. Therefore, there exists
a.s.\ an edge $f \in L$ such that
$\pi(f)\in\Pi_\oo$, whence $\Pi_\oo \ne \es$ a.s.
\end{proof}

We used the fact that $\E_p(\rad(C_0))<\oo$ when $p<\pc$, at the end of the
above proof. Note that the argument cannot be valid when $d=2$ and
$p=\pc=\frac12$, since then $\Pi_\oo=\es$ a.s.

We remark that an alternative proof of Theorem~\ref{sub-crit-surf} proceeds
by applying Lemma~\ref{lem:prel} to the sequence $(W_n)$, where $W_n$ is the
set of sites of $\Z^d$ connected by open paths to $\{v\in\Z^d: v_1=0,
\|v\|\leq n\}$.

\begin{proof}[Proof of Proposition \ref{sep}]
Let $B_n=(-n,n]^d \cap \Z^d$.
Fix $x\in A$, and let $A_n$ be the component of $A\cap B_n$
containing $x$; we set $A_n = \es$ if $x\not\in B_n$.
Let $\Pi_n=\Pi(A_n)$.
By Lemma \ref{lem:prel},
the limit $\Pi_\infty:=\lim_{n\to\infty} \Pi_n$ exists,
and (if non-empty) has empty boundary and only infinite components. Moreover, it is
independent of the choice of $x$ since,
for $x,y\in A$, there exists a path of $A$ joining $x$ to $y$ and,
for all sufficiently large $n$, this path lies in $B_n$.

We argue as follows to show that $\Pi_\oo\ne \es$.
There exists an edge $f=\la a,b\ra$ with $a\in A$ and $b$ joined
to infinity off $A$, so that  $\pi(f) \in \Pi_n$ for all large $n$.
\end{proof}

\section*{Open Questions}

\begin{enumerate}
\item Does $\pc < \psurf$ hold for $\Z^d$ with $3 \le d \le 18$?
\item Does $\pc<\pfin$ hold for $\Z^d$ with $3\leq d\leq 18$?
\item Does $\pfin< \psurf$ hold for $\Z^d$ with $d\geq 3$?
%\item Is there a natural concept of `uniqueness of the infinite surface'?
\end{enumerate}

\section*{Acknowledgements}
GRG acknowledges the hospitality of the Department of
Mathematics at the University of British Columbia, the Section de Math\'ematiques at the University of Geneva,
and the Theory Group at Microsoft Research.
He was supported in part by the Swiss National Science Foundation,
and the EPSRC under grant EP/103372X/1. GK was supported
in part by the Israel Science Foundation.

\bibliography{plaq-final4}

\providecommand{\bysame}{\leavevmode\hbox to3em{\hrulefill}\thinspace}
\providecommand{\MR}{\relax\ifhmode\unskip\space\fi MR }
\providecommand{\MRhref}[2]{%
  \href{http://www.ams.org/mathscinet-getitem?mr=#1}{#2}
}
\providecommand{\href}[2]{#2}
\begin{thebibliography}{10}

\bibitem{ADKS}
D.~Ahlberg, H.~Duminil-Copin, G.~Kozma, and V.~Sidoravicius,
  \emph{Seven-dimensional forest fires},  (2013), available at:
  \url{arXiv:1302.6872}.

\bibitem{AB}
M.~Aizenman and D.~J. Barsky, \emph{Sharpness of the phase transition in
  percolation models}, Commun. Math. Phys. \textbf{108} (1987), 489--526.

\bibitem{ACCFR}
M.~Aizenman, J.~T. Chayes, L.~Chayes, J.~Fr{\"o}hlich, and L.~Russo, \emph{On a
  sharp transition from area law to perimeter law in a system of random
  surfaces}, Commun. Math. Phys. \textbf{92} (1983), 19--69.

\bibitem{BFMSPR}
H.~G. Ballasteros, A.~Fern\'andez, V.~Mart\'{\i}n-Major, M.~Mu\~noz Sudupe,
  G.~Parisi, and J.~J. Ruiz-Lorenzo, \emph{Measures of critical exponents in
  the four dimensional site percolation}, Phys. Lett. B \textbf{400} (1997),
  346--351.

\bibitem{blps-group}
I.~Benjamini, R.~Lyons, Y.~Peres, and O.~Schramm, \emph{Group-invariant
  percolation on graphs}, Geom. Funct. Anal. \textbf{9} (1999), 29--66.

\bibitem{BL91}
B.~Bollob{\'a}s and I.~Leader, \emph{Compressions and isoperimetric
  inequalities}, J. Combin. Theory Ser. A \textbf{56} (1991), 47--62.

\bibitem{BK}
R.~M. Burton and M.~Keane, \emph{Density and uniqueness in percolation},
  Commun. Math. Phys. \textbf{121} (1989), 501--505.

\bibitem{CamR}
M.~Campanino and L.~Russo, \emph{An upper bound on the critical percolation
  probability for the three-dimensional cubic lattice}, Ann. Probab.
  \textbf{13} (1985), 478--491.

\bibitem{DeuP}
J.-D. Deuschel and \'A. Pisztora, \emph{Surface order large deviations for
  high-density percolation}, Probab. Th. Rel. Fields \textbf{104} (1996),
  467--482.

\bibitem{DDGHS}
N.~Dirr, P.~W. Dondl, G.~R. Grimmett, A.~E. Holroyd, and M.~Scheutzow,
  \emph{Lipschitz percolation}, Electron. Commun. Probab. \textbf{15} (2010),
  14--21.

\bibitem{G99}
G.~R. Grimmett, \emph{Percolation}, 2nd ed., Springer, Berlin, 1999.

\bibitem{G-RC}
\bysame, \emph{The {R}andom-{C}luster {M}odel}, Springer, Berlin, 2006.

\bibitem{GrHol}
G.~R. Grimmett and A.~E. Holroyd, \emph{Entanglement in percolation}, Proc.
  London Math. Soc. \textbf{81} (2000), 485--512.

\bibitem{GH10}
\bysame, \emph{Plaquettes, spheres, and entanglement}, Electron. J. Probab.
  \textbf{15} (2010), 1415--1428.

\bibitem{GJ05}
G.~R. Grimmett and S.~Janson, \emph{Branching processes, and random-cluster
  measures on trees}, J. European Math. Soc. \textbf{7} (2005), 253--281.

\bibitem{Hughes2}
B.~D. Hughes, \emph{{Random Walks and Random Environments. Volume 2: Random
  Environments}}, Clarendon Press, Oxford, 1996.

\bibitem{JJ99}
J.~Jonasson, \emph{The random cluster model on a general graph and a phase
  transition characterization of nonamenability}, Stoch. Proc. Appl.
  \textbf{79} (1999), 335--354.

\bibitem{Kes82}
H.~Kesten, \emph{Percolation {T}heory for {M}athematicians}, Birkh\"auser
  Boston, Mass., 1982.

\bibitem{Kes87}
\bysame, \emph{Scaling relations for {2D}-percolation}, Commun. Math. Phys.
  \textbf{109} (1987), 109--156.

\bibitem{KN11}
G.~Kozma and A.~Nachmias, \emph{Arm exponents in high dimensional percolation},
  J. Amer. Math. Soc. \textbf{24} (2011), 375--409.

\bibitem{LMMRS}
L.~Laanait, A.~Messager, S.~Miracle-Sol\'e, J.~Ruiz, and S.~Shlosman,
  \emph{Interfaces in the {P}otts model {I}: {P}irogov--{S}inai theory of the
  {F}ortuin--{K}asteleyn representation}, Commun. Math. Phys. \textbf{140}
  (1991), 81--91.

\bibitem{LSW02}
G.~F. Lawler, O.~Schramm, and W.~Werner, \emph{One-arm exponent for {2D}
  critical percolation}, Electron. J. Probab. \textbf{7} (2002), 1--13.

\bibitem{LSS}
T.~M. Liggett, R.~H. Schonmann, and A.~M. Stacey, \emph{Domination by product
  measures}, Ann. Probab. \textbf{25} (1997), 71--95.

\bibitem{LZ}
C.~D. Lorenz and R.~M. Ziff, \emph{Precise determination of the bond
  percolation thresholds and finite-size scaling corrections for the sc, fcc,
  and bcc lattices}, Phys. Rev. E \textbf{57} (1998), 230--236.

\bibitem{lyons-peres-book}
R.~L{yons with Y.\ Peres}, \emph{{Probability on Trees and Networks}}, 2012, in
  preparation, available at: \url{iu.edu/~rdlyons}.

\bibitem{Men2}
M.~V. Menshikov, \emph{Coincidence of critical points in percolation problems},
  Dokl. Akad. Nauk SSSR \textbf{288} (1986), 1308--1311.

\bibitem{Men1}
M.~V. Menshikov, S.~A. Molchanov, and A.~F. Sidorenko, \emph{Percolation theory
  and some applications}, Probability {T}heory. {M}athematical {S}tatistics.
  {T}heoretical {C}ybernetics, {V}ol. 24 ({R}ussian), Itogi Nauki i Tekhniki,
  Akad. Nauk SSSR, Moscow, 1986, pp.~53--110.

\bibitem{SW01}
S.~Smirnov and W.~Werner, \emph{Critical exponents for two-dimensional
  percolation}, Math. Res. Lett. \textbf{8} (2001), 729--744.

\bibitem{Tas}
H.~Tasaki, \emph{Hyperscaling inequalities for percolation}, Commun. Math.
  Phys. \textbf{113} (1987), 49--65.

\bibitem{Timar}
\'A. Tim\'ar, \emph{Boundary-connectivity via graph theory}, Proc. Amer. Math.
  Soc. \textbf{141} (2013), 475--480.

\bibitem{We07}
W.~Werner, \emph{Lectures on two-dimensional critical percolation},
  {Statistical Mechanics}, IAS/Park City Math. Ser., vol.~16, Amer. Math. Soc.,
  Providence, RI, 2009, pp.~297--360.

\end{thebibliography}
\bibliographystyle{amsplain}
\end{document}